\documentclass[11pt]{amsart}
\usepackage{amsfonts}
\usepackage{amsthm}
\usepackage{geometry}
\usepackage{amsfonts}
\usepackage{amsmath}

\newcommand{\EEQ}{\end{equation}}
\newcommand{\rfb}[1]{\mbox{\rm
   (\ref{#1})}\ifx\undefined\stillediting\else:\fbox{$#1$}\fi}

                         \newcommand{\ud}     {{\rm d}}
\newcommand{\bt}{\begin{Theorem}}
\newcommand{\et}{\end{Theorem}}
\newcommand{\br}{\begin{remark}}
\newcommand{\er}{\end{remark}}
\newcommand{\bc}{\begin{Corollary}}
\newcommand{\ec}{\end{Corollary}}
\newcommand{\el}{\end{Lemma}}
\newcommand{\bd}{\begin{definition}}
\newcommand{\ed}{\end{definition}}

\newcommand{\N}  {\mathbb{N}}
\newcommand{\R}  {\mathbb{R}}

\newcommand{\mm}    {{\hbox{\hskip 0.5pt}}}

\newcommand{\bluff} {{\hbox{\raise 15pt \hbox{\mm}}}}

\newfont{\Blackboard}{msbm10 scaled 1200}

\newfont{\roma}{cmr10 scaled 1200}

\def\CC{\rm \hbox{C\kern-.56em\raise.4ex
         \hbox{$\scriptscriptstyle |$}\kern+0.5 em }}
\newtheorem{corollary}{Corollary}[section]
\newtheorem{definition}[corollary]{Definition}

\newtheorem{remark}[corollary]{Remark}
\numberwithin{equation}{section}
\newtheorem{lem}{Lemma}[section]
\newtheorem{prop}{Proposition}[section]

\newtheorem{cor}{Corollary}[section]
\newtheorem{thm}{Theorem}[section]
\newtheorem{rem}{Remark}[section]

\def\ds{\displaystyle}
\newcommand{\re}{\mathrm{Re}}

\newcommand{\supp}{\mathrm{supp}}
\begin{document}
\thispagestyle{empty}
  
\title[Stability of a star-shaped network with local Kelvin-Voigt damping]{Stability of a star-shaped network with local Kelvin-Voigt damping and non-smooth coefficient at interface}

\author{Fathi Hassine}
\address{UR Analysis and Control of PDEs, UR 13ES64, Department of Mathematics, Faculty of Sciences of Monastir, University of Monastir, Tunisia}
\email{fathi.hassine@fsm.rnu.tn}

\begin{abstract}

In this paper, we study the stability problem of a star-shaped network of elastic strings with a local Kelvin-Voigt damping. Under the assumption that the damping coefficients have some singularities near the transmission point, we prove that the semigroup corresponding to the system is polynomially stable and the decay rates depends on the speed of the degeneracy. This result improves the decay rate of the semigroup associated to the system on an earlier result of Z.~Liu and Q.~Zhang in \cite{LZ} involving the wave equation with local Kelvin-Voigt damping and non-smooth coefficient at interface.
\end{abstract}    

\subjclass[2010]{35B35, 35B40, 93D20}
\keywords{Network of strings, Kelvin-Voigt damping, non-smooth coefficient}

\maketitle

\tableofcontents
 
\section{Introduction}
We consider one-dimensional wave propagation through $N+1$ edges (with $N\geq1$) consisting of an elastic and a Kelvin-Voigt medium all connected to one transmission point. The later material is a viscoelastic material having the properties both of elasticity and viscosity. More precisely we consider the following initial and boundary-value problem
\medskip
\begin{equation}\label{IWKVS1}
\left\{
\begin{array}{ll}
\ds\ddot{u}_{0}(x,t)-u_{0}''(x,t)=0&(x,t)\in(0,\ell_{0})\times(0,+\infty),
\\
\ds\ddot{u}_{j}(x,t)-\left[u_{j}'(x,t)+d_{j}(x)\dot{u}_{j}'(x,t)\right]'=0&(x,t)\in(0,\ell_{1})\times(0,+\infty),\,j=1,\ldots,N,
\\
u_{j}(\ell_{j},t)=0&t\in(0,+\infty),\,j=0,\ldots,N,
\\
u_{0}(0,t)=\dots=u_{N}(0,t)&t\in(0,+\infty),
\\
\ds u_{0}'(0,t)+\sum_{j=1}^{N}u_{j}'(0,t)+d_{j}(0)\dot{u}_{j}'(0,t)=0&t\in(0,+\infty),
\\
\ds u_{j}(x,0)=u_{j}^{0}(x),\;\dot{u}_{j}(x,0)=u_{j}^{1}(x)&x\in(0,\ell_{j}),\,j=0,\ldots,N,
\end{array}
\right.
\end{equation}
where the point stands for the time derivative and the prime stands for the space derivative, $u_{j}:[0,\ell_{j}]\times[0,+\infty[\longrightarrow\mathbb{R}$ for $j=0,\ldots,N$ are the displacement of the of the string of length $\ell_{j}$ and the coefficient damping $d_{j}$ is assumed to be a non-negative function.

The natural energy of system \eqref{IWKVS1} is given by
$$
E(t)=\frac{1}{2}\sum_{j=0}^{N}\int_{0}^{\ell_{j}}\left(|\dot{u}_{j}(x,t)|^{2}+|u_{j}'(x,t)|^{2} \right)\,\ud x
$$
and it is dissipated according to the following law
\begin{equation*}
\frac{\ud}{\ud t}E(t)=-\sum_{j=1}^{N}\int_{0}^{\ell_{j}}d_{j}(x)\,|\dot{u}_{j}'(x,t)|^{2}\,\ud x,\;\forall\,t>0.
\end{equation*}
The stability of this model was intensively studied in this last two decades:

On high dimensional case: Liu and Rao \cite{liu-rao2} proved the exponential decay of the energy providing that the damping region is a neighborhood of the whole boundary, and further restrictions are imposed on the damping coefficient. Next, this result was generalized and improved by Tebou \cite{tibou2} when damping is localized in a suitable open subset, of the domain under consideration, which satisfies the piecewise multipliers condition of Liu. He shows that the energy of this system decays polynomially when the damping coefficient is only bounded measurable, and it decays exponentially when the damping coefficient as well as its gradient are bounded measurable, and the damping coefficient further satisfies a structural condition. Recently, Using Carleman estimates Ammari et al. \cite{AHR2} show a logarithmic decay rate of the semigroup associated to the system when the damping coefficient is arbitrary localized. Next, Burq  generalized this result in \cite{burq2}. This author shows also a polynomial decay rate of the semigroup when the damping region verifying some geometric control condition  of order equal to $\ds\frac{1}{2}$ and of order equal to $\ds\frac{1}{4}$ for a cubic domain with eventual degeneracy on the coefficient damping in case of dimension $2$.

In case of the interval (or when $N=1$): It is well known that the Kelvin-Voigt damping is much stronger than the viscous damping in the sens that if the entire medium is of the Kelvin-Voigt type, the damping for the wave equation not only induces exponential energy decay but also the associated to semigroup is analytic \cite{huang2}; while the entire medium of the viscous type, the associate semigroup is only exponential stable \cite{CFNS} and does not have any smoothing property since the spectrum of the semigroup generator has a vertical asymptote on the left hand side of the imaginary axes of the complex plane. When the damping is localized (i.e., distributed only on the proper subset of the spatial domain), such a comparison is not valid anymore. While the local viscous damping still enjoys the exponential stability as long as the damped region contains an interval of any size in the domain, the local Kelvin–Voigt damping doesn't follow the same analogue. Chen et al. \cite{chen-liu-liu} proved lack of the exponential stability when the damping coefficient is a step function. This unexpected result reveals that the Kelvin–Voigt damping does not follow the ''geometric optics" condition (see \cite{blr}). And Liu and Rao \cite{liu-rao1} proved that the solution of this model actually decays at a rate of $t^{-2}$. The optimality of this order was proven by Alves et al. in \cite{ARSVG}. In 2002, it was shown in \cite{liu-liu2} that exponential energy decay still holds if the damping coefficient is smooth enough. Later, the smoothness condition was weakened in \cite{zhang} and satisfies the following condition
$$
a'(0)=0,\qquad \int_{0}^{x}\frac{|a'(s)|^{2}}{a(s)}\,\ud s\leq C|a'(x)|\quad \forall\,x\in[0,1].
$$
This indicates that the asymptotic behavior of the solution depends on the regularity of the damping coefficient function, which is not the case for the viscous damping model. Renardy \cite{renardy} in 2004 proved that the real part of the eigenvalues are not bounded below if the damping coefficient behaves like $x^{\alpha}$ with $\alpha>1$ near the interface $x=0$. Under this same condition Liu et al. \cite{liu-liu-zhang} proved that the solution of the system is eventually differentiable which also guarantees the exponential stability since there is no spectrum on the imaginary axis and the system is dissipative. In 2016, Under the assumption that the damping coefficient has a singularity at the interface of the damped and undamped regions and behaves like $x^{\alpha}$ with $\alpha\in(0,1)$ near the interface, \cite{LZ} proved that the semigroup corresponding to the system is polynomially stable and the decay rate depends on the parameter $\alpha$.

In case of multi-link structure: In \cite{hassine1} we proved that the semigroup is polynomially stable when the coefficient damping is piecewise function with an optimal decay rate equal to $2$ (we refer also to \cite{hassine2} for the case of transmission Euler-Bernoulli plate and wave equation with a localized Kelvin-Voigt damping). Recently, Ammari et al. \cite{ALS} consider the tree of elastic strings with local Kelvin-Voigt damping. They proved under some assumptions on the smoothness of the damping coefficients, say $W^{2,\infty}$, and some other considerations that the semigroup is exponentially stable if the damping coefficient is continuous at every node of the tree and otherwise it is  polynomially stable with a decay rate equal to $2$.

In light of all the above results it is obvious to say that the asymptotic behavior of the solution to system \eqref{IWKVS1} depends on the regularity on the damping coefficient function. In this paper we want to generalize and improve the polynomial decay rate given by Liu and Zhang in \cite{LZ} when the damping coefficient has a singularity at the interface and behaves like $x^{\alpha}$ with $\alpha\in(0,1)$. Precisely, we make the following assumptions: For every $j=1,\ldots,N$
\begin{itemize}
	\item[$\bullet$] There exist $a_{j},\,b_{j}\in[0,\ell_{j}]$ with $a_{j}<b_{j}$ such that 
	\begin{equation}\label{IWKVS4}
	[a_{j},b_{j}]\subset\supp(d_{j})\qquad\text{ and }\qquad d_{j}\in L^{\infty}(0,\ell_{j}).\tag{A \addtocounter{equation}{1}\theequation}
	\end{equation}
	\item[$\bullet$] There exist $\alpha_{j}\in(0,1)$ and $\kappa_{j}\geq 0$ such that
	\begin{equation}\label{IWKVS2}
	\lim_{x\to 0^{+}}\frac{d_{j}(x)}{x^{\alpha_{j}}}=\kappa_{j}.\tag{A \addtocounter{equation}{1}\theequation}
	\end{equation}
	\item There exists $\eta_{j}\in[0,1)$ such that
	\begin{equation}\label{IWKVS3}
	\lim_{x\to 0^{+}}\frac{xd_{j}'(x)}{d_{j}(x)}=\eta_{j}.\tag{A \addtocounter{equation}{1}\theequation}
	\end{equation}
\end{itemize}
\begin{rem}\label{IWKVS6}
The typical example of functions $d_{j}$ that satisfies assumptions \eqref{IWKVS4}, \eqref{IWKVS2} and \eqref{IWKVS3}, is when $d_{j}(x)=x^{\alpha_{j}}$ with $\alpha_{j}\in(0,1)$ for every $j=1,\ldots,N$. An interesting example too is when we take $d_{j}(x)=x^{\alpha_{j}'}|\ln(x)|^{\beta_{j}}$ with $\alpha_{j}'\in(0,1)$ and $\beta_{j}>0$ in this case  assumptions \eqref{IWKVS4}, \eqref{IWKVS2} and \eqref{IWKVS3} are satisfied with $\kappa_{j}=0$, $\alpha_{j}=\alpha_{j}'-\varepsilon_{j}$ for all $0<\varepsilon_{j}<\alpha_{j}'$ and $\eta_{j}=\alpha_{j}'$.
\end{rem}

Let $\ds\mathcal{H}= V\times\prod_{j=0}^{N}L^{2}(0,\ell_{j})$ be the Hilbert space endowed with the inner product define for $(u,v)=\left((u_{j})_{j=0,\ldots,N},(v_{j})_{j=0,\ldots,N}\right)\in\mathcal{H}$ and $(\tilde{u},\tilde{v})=\left((\tilde{u}_{j})_{j=0,\ldots,N},(\tilde{v}_{j})_{j=0,\ldots,N}\right)\in\mathcal{H}$ by
$$ 
\left\langle (u,v),(\tilde{u},\tilde{v})\right\rangle_{\mathcal{H}}=\sum_{j=0}^{N}\int_{0}^{\ell_{j}}u_{j}'(x)\,.\,\overline{\tilde{u}}_{j}'(x)\,\ud x+\int_{0}^{\ell_{j}}v_{j}(x)\,.\,\overline{\tilde{v}}_{j}(x)\,\ud x,
$$
where $V$ is the Hilbert space defined by
$$
V=\left\{(u_{j})_{j=0,\ldots,N}\in\prod_{j=0}^{N}H^{1}(0,\ell_{j}): u_{j}(\ell_{j})=0\;\forall j=0,\ldots,N;\;u_{0}(0)=\dots=u_{N}(0)\right\}.
$$
By setting $U(t)=\left((u_{j}(t))_{j=0,\ldots,N},(v_{j}(t))_{j=0,\ldots,N}\right)$ and $U^{0}=\left((u_{j}^{0})_{j=0,\ldots,N},(u_{j}^{1})_{j=0,\ldots,N}\right)$ we can rewrite system \eqref{IWKVS1} as a first order differential equation as follows
\begin{equation}\label{WPWKVS1}
\dot{U}(t)=\mathcal{A}U(t),\qquad U(0)=U^{0}\in\mathcal{D}(\mathcal{A}),
\end{equation}
where
$$
\mathcal{A}\left((u_{j})_{j=0,\ldots,N},(v_{j})_{j=0,\ldots,N}\right)=\left((v_{j})_{j=0,\ldots,N},(u_{0}'',[u_{1}'+d_{1}v_{1}']',\ldots,[u_{N}'+d_{N}v_{N}']')\right),
$$
with
\begin{align*}
\mathcal{D}(\mathcal{A})=\bigg\{\left((u_{j})_{j=0,\ldots,N},(v_{j})_{j=0,\ldots,N}\right)\in\mathcal{H}:\;(v_{j})_{j=0,\ldots,N}\in V;\;u_{0}''\in L^{2}(0,\ell_{0});
\\
[u_{j}'+d_{j}v_{j}']'\in L^{2}(0,\ell_{j})\,\forall\,j=1,\ldots,N;\;u_{0}'(0)+\sum_{j=1}^{N}u_{j}'(0)+d_{j}(0)v_{j}'(0)=0\bigg\}. 
\end{align*}
\begin{thm}\label{IWKVS5}
Assume that for $j=1,\ldots,N$ the coefficient functions $d_{j}\in\mathcal{C}([0,\ell_{j}])\cap\mathcal{C}^{1}((0,\ell_{j}))$ are such that conditions \eqref{IWKVS4}, \eqref{IWKVS2} and \eqref{IWKVS3} hold. Then, the semigroup $e^{t\mathcal{A}}$ associated to system \eqref{IWKVS1} (see Proposition \ref{WPWKVS4}) is polynomially stable precisely we have: There exists $C>0$ such that
$$
\|e^{t\mathcal{A}}U^{0}\|_{\mathcal{H}}\leq\frac{C}{(t+1)^{\frac{2-\alpha}{1-\alpha}k}}\left\|U^{0}\right\|_{\mathcal{D}(\mathcal{A}^{k})}\quad\forall\,U^{0}=\left((u_{j}^{0})_{j=0,\ldots,N},(u_{j}^{1})_{j=0,\ldots,N}\right)\in\mathcal{D}(\mathcal{A}^{k})\;\forall\,t\geq 0,
$$
where $\ds\alpha=\min\{\alpha_{1},\ldots,\alpha_{N}\}$.
\end{thm}
\begin{rem}
This theorem reveals that the stability order of the semigroup $e^{t\mathcal{A}}$ associated to problem \eqref{IWKVS1} depends on the behavior of the damping coefficients $d_{j}$ described by the parameters $\alpha_{j}$ for $j=1,\dots,N$. This result improves the decay rate of the energy given in \cite{LZ} from $\ds\frac{1}{1-\alpha}$ to $\ds\frac{2-\alpha}{1-\alpha}$ and make it more meaningful in fact, as $\alpha$ goes to $1^{-}$ the order of polynomial stability $\ds\frac{2-\alpha}{1-\alpha}$ goes to $\infty$ which is consistent with the exponential stability when $\alpha=1$ (see \cite{ALS,LZ,liu-liu2,LZ}) and as $\alpha$ goes to $0^{+}$ the order of polynomial stability $\ds\frac{2-\alpha}{1-\alpha}$ goes to $2$ which is consistent with the optimal order stability when $\alpha=0$ (see \cite{ARSVG,hassine1,liu-rao1}).
\end{rem}
\begin{rem}
When the coefficient functions $d_{j}$ behave polynomially near $0$ as $x^{\alpha_{j}}$ with $0<\alpha_{j}<1$ then from Theorem \ref{IWKVS5} the semigroup $e^{t\mathcal{A}}$ decays polynomially with the decay rate given above in the theorem. Moreover, when $d_{j}$ decay faster than $x^{\alpha_{j}}$ and slower than $x^{\alpha_{j}+\varepsilon}$ for all $\varepsilon>0$ this can may be seen for instance with the example of $d_{j}(x)=x^{\alpha_{j}}|\ln(x)|^{\beta_{j}}$ with $\beta_{j}>0$ then according to Remark \ref{IWKVS6} the semigroup $e^{t\mathcal{A}}$ decays polynomially with the decay rate equal to $\ds\frac{2-\alpha+\varepsilon}{1-\alpha+\varepsilon}$ for each $\varepsilon\in(0,\alpha)$ where $\ds\alpha=\min\{\alpha_{1},\ldots,\alpha_{N}\}$. Consequently, this decay rate is worse than if the coefficient functions $d_{j}$ behave near $0$ like $x^{\alpha_{j}}$ and better than if the coefficient functions $d_{j}$ behave near $0$ like $x^{\alpha_{j}'}$ for any $\alpha_{j}'>0$ such that $\alpha_{j}'<\alpha_{j}$.
\end{rem}

This article is organized as follows. In section \ref{WPWKVS}, we prove the well posedness of system \eqref{IWKVS1}. In section \ref{SSWKVS}, we show that the semigroup associated to the generator $\mathcal{A}$ is strongly stable . In section \ref{PSWKVS}, we prove the polynomial decay rate given by Theorem \ref{IWKVS5}.
\section{Well-posedness}\label{WPWKVS}
In this section we use the semigroup approach to prove the well-posedness of system \eqref{IWKVS1}.
\begin{prop}\label{WPWKVS4}
Assume that condition \eqref{IWKVS4} holds. Then $\mathcal{A}$ generates a $C_{0}$-semigroup of contractions $e^{t\mathcal{A}}$ on the Hilbert space $\mathcal{H}$.
\end{prop}
\begin{proof}
For any $(u,v)=\big((u_{j})_{j=0,\ldots,N},(v_{j})_{j=0,\ldots,N}\big)\in\mathcal{D}(\mathcal{A})$, we have
$$
\langle\mathcal{A}(u,v),(u,v)\rangle_{\mathcal{H}}=-\sum_{j=1}^{N}\int_{0}^{\ell_{j}}d_{j}(x)|\partial_{x}v_{j}(x)|^{2}\,\ud x.
$$
This shows that the operator $\mathcal{A}$ is dissipative.
\\
Given $(f,g)=\big((f_{j})_{j=0,\ldots,N},(g_{j})_{j=0,\ldots,N}\big)\in\mathcal{H}$ we look for $(u,v)=\big((u_{j})_{j=0,\ldots,N},(v_{j})_{j=0,\ldots,N}\big)\in\mathcal{D}(\mathcal{A})$ such that $\mathcal{A}(u,v)=(f,g)$, this is also written
\begin{equation*}
\left\{\begin{array}{l}
v_{j}=f_{j}\quad\forall\, j=0,\ldots,N,
\\
u_{0}''=g_{0},
\\
(u_{j}'+d_{j}v_{j}')'=g_{j}\quad\forall\, j=1,\ldots,N,
\\
u_{j}(\ell_{j})=0\quad\forall\, j=0,\ldots,N,
\\
u_{0}(0)=\dots=u_{N}(0),
\\
\ds u_{0}'(0)+\sum_{j=1}^{N}u_{j}'(0)+d_{j}(0) v_{j}'(0)=0,
\end{array}\right.
\end{equation*}
or equivalently
\begin{equation}\label{WPWKVS3}
\left\{\begin{array}{l}
v_{j}=f_{j}\quad\forall\, j=0,\ldots,N,
\\
u_{0}''=g_{0},
\\
u_{j}''=g_{j}-(d_{j}f_{j}')'\quad\forall\, j=1,\ldots,N,
\\
u_{j}(\ell_{j})=0\quad\forall\, j=0,\ldots,N,
\\
u_{0}(0)=\dots=u_{N}(0),
\\
\ds u_{0}'(0)+\sum_{j=1}^{N}u_{j}'(0)+d_{j}(0)f_{j}'(0)=0.
\end{array}\right.
\end{equation}
For this aim we set the continuous coercive and bi-linear form in $V$
$$
L(u,\tilde{u})=\sum_{j=0}^{N}\int_{0}^{\ell_{j}}u_{j}'\,.\,\overline{\tilde{u}}_{j}'\,\ud x.
$$
By Lax-Milligram theorem there exists a unique element $(u_{j})_{j=0,\ldots,N}\in V$ such that
\begin{equation}\label{WPWKVS2}
\sum_{j=0}^{N}\int_{0}^{\ell_{j}}u_{j}'\,.\,\overline{\tilde{u}}_{j}'\,\ud x=-\sum_{j=1}^{N}\int_{0}^{\ell_{j}}d_{j}f_{j}'\,.\,\overline{\tilde{u}}_{j}'\,\ud x-\sum_{j=0}^{N}\int_{0}^{\ell_{j}}g_{j}\,.\,\overline{\tilde{u}}_{j}\,\ud x.
\end{equation}
It follows that by taking \eqref{WPWKVS2} in the sens of distribution that $\partial_{x}^{2}u_{0}=g_{0}$ in $L^{2}(0,\ell_{0})$ and $u_{j}''=g_{j}-[d_{j}f_{j}']'$ in $L^{2}(0,\ell_{j})$ for all $j=1,\ldots,N$. Back again to \eqref{WPWKVS2} and integrating by parts we find that $\ds u_{0}'(0)+\sum_{j=1}^{N}u_{j}'(0)+d_{j}(0)f_{j}'(0)=0$. This prove that the operator $\mathcal{A}$ is surjective. Moreover, by multiplying the second line by $\overline{u}_{1}$ of \eqref{WPWKVS3} and the third line by $\overline{u}_{j}$ and integrating  over $(0,\ell_{0})$ and $(0,\ell_{j})$ respectively and summing up then by Poincar\'e inequality and Cauchy-Schwarz inequality we find that there exists a constant $C>0$ such that
$$
\sum_{j=0}^{N}\int_{0}^{\ell_{j}}|u_{j}'|^{2}\ud x\leq C\left(\sum_{j=0}^{N}\int_{0}^{\ell_{j}}|f_{j}'|^{2}\ud x+\sum_{j=0}^{N}\int_{0}^{\ell_{j}}|g_{j}|^{2}\ud x\right)
$$
which combined with the first line of \eqref{WPWKVS3} leads to 
$$
\sum_{j=0}^{N}\int_{0}^{\ell_{j}}|u_{j}'|^{2}\ud x+\sum_{j=0}^{N}\int_{0}^{\ell_{j}}|v_{j}|^{2}\ud x\leq C\left(\sum_{j=0}^{N}\int_{0}^{\ell_{j}}|f_{j}'|^{2}\ud x+\sum_{j=0}^{N}\int_{0}^{\ell_{j}}|g_{j}|^{2}\ud x\right).
$$
This implies that $0\in\rho(\mathcal{A})$ and by contraction principle, we easily get $R(\lambda-\mathcal{A})=\mathcal{H}$ for sufficient small $\lambda>0$. Since $\mathcal{D}(\mathcal{A})$ is dense in $\mathcal{H}$ then thanks to Lumer-Phillips theorem \cite[Theorem 1.4.3]{Pazy}, $\mathcal{A}$ generates a $C_{0}$-semi-group of contractions on $\mathcal{H}$.
\end{proof}
As a consequence of Proposition \ref{WPWKVS4} we have the following well-posedness result of system \eqref{IWKVS1}.
\begin{cor}
For any initial data $U^{0}\in\mathcal{H}$, there exists a unique solution $U(t)\in\mathcal{C}([0,\,+\infty[,\,\mathcal{H})$ to the  problem \eqref{WPWKVS1}. Moreover, if $U^{0}\in\mathcal{D}(\mathcal{A})$, then
$$
U(t)\in\mathcal{C}([0,\,+\infty[,\, \mathcal{D}(\mathcal{A}))\cap\mathcal{C}^{1}([0,\,+\infty),\, \mathcal{H}).
$$
\end{cor}
\section{Strong stability}\label{SSWKVS}
The aim of this section is to prove that the semi-group generated by the operator $\mathcal{A}$ is strongly stable. In another words this means that the energy of system \eqref{IWKVS1} degenerates over the time to zero.
\begin{lem}\label{SSWKVS4}
Assume that condition \eqref{IWKVS4} holds. Then for every $\lambda\in\R$ the operator $(i\lambda-\mathcal{A})$ is injective.
\end{lem}
\begin{proof}
Since $0\in\rho(\mathcal{A})$ (according to the proof of Theorem \ref{WPWKVS2}), we only need to check that for every $\lambda\in\R^{*}$ we have $\ker(i\lambda I-\mathcal{A})=\{0\}$. Let $\lambda\neq 0$ and $(u,v)=\big((u_{j})_{j=0,\ldots,N},(v_{j})_{j=0,\ldots,N}\big)\in\mathcal{D}(\mathcal{A})$ such that 
\begin{equation}\label{SSWKVS6}
\mathcal{A}(u,v)=i\lambda (u,v)
\end{equation}
Taking the real part of the inner product in $\mathcal{H}$ of \eqref{SSWKVS6} with $(u,v)$ and using the dissipation of $\mathcal{A}$, we get
\begin{equation*}
\re\,i\lambda\|(u,v)\|_{\mathcal{H}}^{2}=\re\,\langle\mathcal{A}(u,v),(u,v)\rangle_{\mathcal{H}}=-\sum_{j=1}^{N}\int_{0}^{\ell_{j}}d_{j}(x)|v_{j}'(x)|^{2}\,\ud x=0,
\end{equation*}
which implies that
\begin{equation}\label{SSWKVS1}
d_{j}v_{j}'=0\quad \text{in }L^{2}(0,\ell_{j}),\; \forall\, j=1,\ldots,N.
\end{equation}
Inserting \eqref{SSWKVS1} into \eqref{SSWKVS6}, we obtain
\begin{equation}\label{SSWKVS2}
\left\{\begin{array}{ll}
i\lambda u_{j}=v_{j}&\text{in }(0,\ell_{j}),\quad j=0,\ldots,N,
\\
\lambda^{2} u_{j}+u_{j}''=0&\text{in }(0,\ell_{j}),\quad j=0,\ldots,N,
\\
u_{j}(\ell_{j})=0&j=0,\ldots,N,
\\
u_{1}(0)=\dots=u_{N}(0)&
\\
\ds\sum_{j=0}^{N}u_{j}'(0)=0&
\end{array}\right.
\end{equation}
Combining \eqref{SSWKVS1} with the first line of \eqref{SSWKVS2}, we get
$$
u_{j}'=0 \quad\text{a.e in } [a_{j},b_{j}]\quad\forall\,j=1,\ldots,N.
$$
Since $u_{j}\in H^{2}(0,\ell_{j})$ then following to the embedding $H^{1}(0,\ell_{j})\hookrightarrow C^{0}(0,\ell_{j})$ we have
\begin{equation}\label{SSWKVS3}
u_{j}'\equiv0\quad\text{in } [a_{j},b_{j}]\quad\forall\,j=1,\ldots,N.
\end{equation}
Which by the second line of \eqref{SSWKVS2} leads to 
$$
u_{j}\equiv 0\quad\text{in } [a_{j},b_{j}]\quad\forall\,j=1,\ldots,N.
$$
So for every $j=1,\ldots,N$, $u_{j}$ and $v_{j}$ are solution to the following problem
$$
\left\{\begin{array}{ll}
i\lambda u_{j}=v_{j}&\text{in }(0,\ell_{j}),
\\
\lambda^{2} u_{j}+u_{j}''=0&\text{in }(0,\ell_{j}),
\\
u_{j}(a_{j})=u_{j}'(a_{j})=0,
\end{array}\right.
$$
and this clearly gives that $u_{j}=v_{j}\equiv 0$ in $(0,\ell_{j})$ for every $j=1,\ldots,N$. Following to system \eqref{SSWKVS2} we have then
\begin{equation}\label{SSWKVS9}
\left\{\begin{array}{ll}
i\lambda u_{0}=v_{0}&\text{in }(0,\ell_{0}),
\\
\lambda^{2} u_{0}+u_{0}''=0&\text{in }(0,\ell_{0})
\\
u_{0}(\ell_{0})=u_{0}'(\ell_{0})=0&
\end{array}\right.
\end{equation}
which gives  also that $u_{0}=v_{0}\equiv 0$ in $(0,\ell_{0})$. This shows that $(u,v)=(0,0)$ and consequently $(i\lambda I-\mathcal{A})$ is injective for all $\lambda\in\R$.
\end{proof}
\begin{lem}\label{SSWKVS5}
Assume that condition \eqref{IWKVS4} holds. Then for every $\lambda\in\R$ the operator $(i\lambda-\mathcal{A})$ is surjective.
\end{lem}
\begin{proof}
Since $0\in\rho(\mathcal{A})$ (see proof of Theorem \ref{WPWKVS2}), we only need to check that for every $\lambda\in\R^{*}$ we have $R(i\lambda I-\mathcal{A})=\mathcal{H}$. Let $\lambda\neq 0$, then given $(f,g)=\big((f_{j})_{j=0,\ldots,N},(g_{j})_{j=0,\ldots,N}\big)\in\mathcal{H}$ we are looking for $(u,v)=\big((u_{j})_{j=0,\ldots,N},(v_{j})_{j=0,\ldots,N}\big)\mathcal{D}(\mathcal{A})$ such that 
\begin{equation}\label{SSWKVS7}
(i\lambda I-\mathcal{A})(u,v)=(f,g),
\end{equation}
or equivalently
\begin{equation}\label{SSWKVS8}
\left\{\begin{array}{ll}
v_{j}=i\lambda u_{j}-f_{j}&\text{in }(0,\ell_{j}),\quad j=0,\ldots,N,
\\
-\lambda^{2} u_{0}-u_{0}''=i\lambda f_{0}+g_{0}&\text{in }(0,\ell_{0}),
\\
-\lambda^{2} u_{j}-(u_{j}'+i\lambda d_{j}u_{j'})'=i\lambda f_{j}+g_{j}-(d_{j}f_{j}')'&\text{in }(0,\ell_{j}),\quad j=1,\ldots,N.
\end{array}\right.
\end{equation}
We define  for all $u=((u_{j})_{j=0,\ldots,N}\big)\in V$ the operator
$$
Au=\big(-u_{0}'',-(u_{1}'+i\lambda d_{j}u_{1}')',\ldots,-(u_{N}'+i\lambda d_{N} u_{N}')'\big).
$$
Thanks to Lax-Milgram's theorem \cite[Theorem 2.9.1]{LM}, it is easy to show that $A$ is an isomorphism from $V$ into $V'$ (where $V'$ is the dual space of $V$ with respect to the pivot space $H$). Then the second ant the third line of \eqref{SSWKVS8} can be written as follows
\begin{equation}\label{SSWKVS12}
u-\lambda^{2}A^{-1}u=A^{-1}\big(i\lambda f_{0}+g_{0},i\lambda f_{1}+g_{1}-(d_{1}f_{1}')',\ldots,i\lambda f_{N}+g_{N}-(d_{N}f_{N}')'\big).
\end{equation}
If $u\in\ker(I-\lambda A^{-1})$, then we obtain
\begin{equation}\label{SSWKVS10}
\left\{\begin{array}{ll}
\lambda^{2}u_{0}+u_{0}''=0&\text{in }(0,\ell_{0}),
\\
\lambda^{2}u_{j}+(u_{j}'+i\lambda d_{j}u_{j}')'=0&\text{in }(0,\ell_{j}),\quad j=1,\ldots,N.
\end{array}\right.
\end{equation}
For $j=1,\ldots,N$ we multiply each line of \eqref{SSWKVS10} by $\overline{u}_{j}$ and integrating over $(0,\ell_{j})$ and summing up
\begin{equation}\label{SSWKVS11}
\lambda^{2}\sum_{j=0}^{N}\int_{0}^{\ell_{j}}|u_{j}|^{2}\,\ud x-\sum_{j=0}^{N}\int_{0}^{\ell_{j}}|u_{j}'|^{2}\,\ud x-i\lambda\sum_{j=1}^{N}\int_{0}^{\ell_{j}}d_{j}|u_{j}'|^{2}\,\ud x=0.
\end{equation}
By taking the imaginary part of \eqref{SSWKVS11} we get
$$
\sum_{j=1}^{N}\int_{0}^{\ell_{j}}d_{j}|u_{j}'|^{2}\,\ud x=0.
$$
This means that $d_{j}u_{j}'=0$ in $(0,\ell_{j})$ for all $i=1,\ldots,N$, which inserted into \eqref{SSWKVS10}  one gets
$$
\lambda^{2}u_{j}+u_{j}''=0\quad\text{in }(0,\ell_{j}),\quad j=0,\ldots,N.
$$
Then using the same arguments as proof of Lemma \ref{SSWKVS4} we find that $u=0$. Hence, we proved that $\ker(I-\lambda^{2}A^{-1})=\{0\}$. Besides, thanks to the compact embeddings $V\hookrightarrow H$ and $H\hookrightarrow V'$ the operator $A^{-1}$ is compact in $V$. So that, following to Fredholm's alternative, the operator $(I-\lambda^{2}A^{-1})$ is invertible in $V$. Therefore, equation \eqref{SSWKVS12} have a unique solution in $V$. Thus, the operator $i\lambda I-\mathcal{A}$ is surjective. This completes the proof.
\end{proof}
Thanks to Lemmas \ref{SSWKVS4} and \ref{SSWKVS5} and the closed graph theorem we have $\sigma(\mathcal{A})\cap i\R=\emptyset$. This with Arendt and Batty \cite{AB} result following to which a $C_{0}$-semi-group of contractions in a Banach space is strongly stable, if $\rho(\mathcal{A})\cap i\R$ contains only a countable number of continuous spectrum of $\mathcal{A}$ lead to the following
\begin{thm}
Assume that condition \eqref{IWKVS4} holds. Then the semigroup $(e^{t \mathcal{A}})_{t \geq 0}$ is strongly stable in the energy space $\mathcal{H}$ i.e.,
$$
\lim_{t\to+\infty}\|e^{t \mathcal{A}}U^{0}\|_{\mathcal{H}}=0,\quad \forall\,U^{0}\in\mathcal{H}.
$$
\end{thm}
\section{Polynomial stability}\label{PSWKVS}
In this section, we prove Theorem \ref{IWKVS5}. The idea is to estimate the energy norm and boundary terms at the interface by the local viscoelastic damping. The difficulty is to deal with the higher order boundary term at the interface so that the energy on $(0,\ell_{0})$ can be controlled by the viscoelastic damping on $(0,\ell_{j})$ for every $j=1,\ldots, N$. Our proof is based on the following result
\begin{prop}\cite[Theorem 2.4]{borichevtomilov}\label{PSWKVS2}
Let $e^{tB}$ be a bounded $C_{0}$-semi-group on a Hilbert space $X$ with generator $B$ such that $i\R\in\rho(A)$. Then $e^{tB}$ is polynomially stable with order $\ds\frac{1}{\gamma}$ i.e. there exists $C>0$ such that
$$
\|e^{tB}u\|_{X}\leq \frac{C}{t^{\frac{1}{\gamma}}}\|u\|_{\mathcal{D}(B)}\quad\forall\,u\in\mathcal{D}(B)\;\forall\, t\geq 0,
$$
if and only if
$$
\limsup_{|\lambda|\rightarrow\infty}\|\lambda^{-\gamma}(i\lambda I-B)^{-1}\|_{X}<\infty.
$$
\end{prop}
According to Proposition \ref{PSWKVS2} we shall verify that for $\alpha=\min\{\alpha_{1},\ldots,\alpha_{N}\}$ and $\ds\gamma=\frac{1-\alpha}{2-\alpha}$ there exists $C_{0}>0$ such that
\begin{equation}\label{PSWKVS3}
\inf_{\substack{\|((u_{j})_{j=0,\ldots,N},(v_{j})_{j=0,\ldots,N})\|_{\mathcal{H}}=1
\\
\lambda\in\R}}\lambda^{\gamma}\left\|i\lambda\left((u_{j})_{j=0,\ldots,N},(v_{j})_{j=0,\ldots,N}\right)-\mathcal{A}\left((u_{j})_{j=0,\ldots,N},(v_{j})_{j=0,\ldots,N}\right)\right\|_{\mathcal{H}}\geq C_{0}.
\end{equation}
Suppose that \eqref{PSWKVS3} fails then there exist a sequence of real numbers $\lambda_{n}$ and a sequence of functions $(u_{n},v_{n})_{n\in\N}=\big((u_{0,n},\ldots,u_{N,n}),(v_{0,n},\ldots,v_{N,n})\big)_{n\in\N}\subset\mathcal{D}(\mathcal{A})$ such that
\begin{eqnarray}
&\lambda_{n}\,\longrightarrow\,\infty\;\text{ as }\;n\,\longrightarrow\,\infty,&
\\
&\big\|(u_{n},v_{n})\big\|=1,&\label{PSWKVS4}
\\
&\lambda_{n}^{\gamma}\left\|i\lambda_{n}(u_{n},v_{n})-\mathcal{A}(u_{n},v_{n})\right\|_{\mathcal{H}}=o(1).&\label{PSWKVS9}
\end{eqnarray}
Since, we have
\begin{equation}\label{PSWKVS10}
\lambda_{n}^{\gamma}\re\langle i\lambda (u_{n},v_{n})-\mathcal{A}(u_{n},v_{n}),(u_{n},v_{n})\rangle=\lambda_{n}^{\gamma}\sum_{j=1}^{n}\int_{0}^{\ell_{j}}d_{j}|v_{j,n}'|^{2}\,\ud x
\end{equation}
then using \eqref{PSWKVS4} and \eqref{PSWKVS9} we obtain
\begin{equation}\label{PSWKVS11}
\sum_{j=1}^{n}\|d_{j}^{\frac{1}{2}}v_{j,n}'\|_{L^{2}(0,\ell_{j})}=o(\lambda_{n}^{-\frac{\gamma}{2}}).
\end{equation}
Following to \eqref{PSWKVS9} we have
\begin{align}
\lambda_{n}^{\gamma}(i\lambda_{n} u_{j,n}-v_{j,n})&=f_{j,n}\,\longrightarrow\,0\quad\text{in } H^{1}(0,\ell_{j}),\;j=0,\ldots,N,\label{PSWKVS5}
\\
\lambda_{n}^{\gamma}(i\lambda_{n} v_{0,n}-u_{0,n}'')&=g_{0,n}\,\longrightarrow\,0\quad\text{in } L^{2}(0,\ell_{0}),\label{PSWKVS6}
\\
\lambda_{n}^{\gamma}(i\lambda_{n} v_{j,n}-T_{j,n}')&=g_{j,n}\,\longrightarrow\,0\quad\text{in } L^{2}(0,\ell_{j}),\;j=1,\ldots,N,\label{PSWKVS7}
\end{align}
with the transmission conditions
\begin{align}
u_{0,n}(0)=\dots=u_{N,n}(0),\label{PSWKVS57}
\\
u_{0,n}'(0)+\sum_{j=1}^{N}T_{j,n}(0)=0,\label{PSWKVS58}
\end{align}
where for $j=1,\ldots,N$ we have denoted by
\begin{equation}\label{PSWKVS8}
T_{j,n}=u_{j,n}'+d_{j}v_{j,n}'=(1+i\lambda_{n}d_{j})u_{j,n}'-\lambda_{n}^{-\gamma}d_{j}f_{j,n}'.
\end{equation}
By \eqref{PSWKVS11} and \eqref{PSWKVS5} we find
\begin{equation}\label{PSWKVS12}
\sum_{j=1}^{n}\|d_{j}^{\frac{1}{2}}u_{j,n}'\|_{L^{2}(0,\ell_{j})}=o(\lambda_{n}^{-\frac{\gamma}{2}-1}).
\end{equation}
One multiplies \eqref{PSWKVS6} by $(x-\ell_{0})\overline{u}_{0,n}'$ integrating by parts over the interval $(0,\ell_{0})$ and use \eqref{PSWKVS4} and \eqref{PSWKVS5} we get
\begin{equation}\label{PSWKVS13}
\int_{0}^{\ell_{0}}\left(|u_{0,n}'|^{2}+|v_{0,n}|^{2}\right)\,\ud x-\ell_{0}\left(|u_{0,n}'(0)|^{2}+|v_{0,n}(0)|^{2}\right)=o(1).
\end{equation}
We multiply \eqref{PSWKVS7} by $v_{j,n}$ for $j=1,\ldots,N$ and \eqref{PSWKVS6} by $v_{0,n}$ then integrating over $(0,\ell_{j})$  for $j=0,\ldots,N$ and summing up to get
\begin{equation}\label{PSWKVS14}
i\lambda_{n}^{\gamma+1}\sum_{j=0}^{N}\|v_{j,n}\|_{L^{2}(0,\ell_{j})}^{2}+\lambda_{n}^{\gamma}\sum_{j=0}^{N}\langle u_{j,n}',v_{j,n}'\rangle_{L^{2}(0,\ell_{j})}+\lambda_{n}^{\gamma}\sum_{j=1}^{N}\|d_{j}^{\frac{1}{2}}v_{j,n}'\|_{L^{2}(0,\ell_{j})}=o(1).
\end{equation}
We take the inner product of \eqref{PSWKVS5} with $u_{j,n}$ in $H^{1}(0,\ell_{j})$ for $j=0,\ldots,N$ and summing up,
\begin{equation}\label{PSWKVS15}
i\lambda_{n}^{\gamma+1}\sum_{j=0}^{N}\|u_{j}'\|_{L^{2}(0,\ell_{j})}^{2}-\lambda_{n}^{\gamma}\sum_{j=0}^{N}\langle v_{j,n}',u_{j,n}'\rangle_{L^{2}(0,\ell_{j})}=o(1).
\end{equation}
Adding \eqref{PSWKVS14} and \eqref{PSWKVS15} and taking the imaginary part of the equality then by \eqref{PSWKVS11} we arrive at
\begin{equation}\label{PSWKVS17}
\sum_{j=0}^{N}\left(\|u_{j,n}'\|_{L^{2}(0,\ell_{j})}-\|v_{j,n}\|_{L^{2}(0,\ell_{j})}\right)=o(1).
\end{equation}
At this stage we recall the following Hardy type inequalities
\begin{lem}\cite[Theorem 3.8]{MV}\label{PSWKVS16}
Let $L>0$ and $a:[0,L]\rightarrow\R_{+}$ be such that $a\in\mathcal{C}([0,L])\cap\mathcal{C}^{1}((0,L])$ and satisfying
$$
\lim_{x\rightarrow+\infty}\frac{x a'(x)}{a(x)}=\eta\in[0,1).
$$
Then there exists $C(\eta,L)>0$ such that for all locally continuous function $z$ on $[0,L]$ satisfying
$$
z(0)=0\quad\text{ and }\quad\int_{0}^{L}a(x)|z'(x)|^{2}\,\ud x<\infty
$$
the following inequality holds
$$
\int_{0}^{L}\frac{a(x)}{x^{2}}|z(x)|^{2}\,\ud x\leq C(\eta,L)\int_{0}^{L}a(x)|z'(x)|^{2}\,\ud x.
$$
\end{lem}
\begin{lem}\cite[Lemma 2.2]{LZ}\label{PSWKVS18}
Let $L>0$ and $\rho_{1},\,\rho_{2}>0$ be two weight functions defied on $(0,L)$. Then the following conditions are equivalent:
\begin{equation}\label{PSWKVS19}
\int_{0}^{L}\rho_{1}(x)|Tf(x)|^{2}\,\ud x\leq C\int_{0}^{L}\rho_{2}(x)|f(x)|^{2}\,\ud x,
\end{equation}
and
$$
K=\sup_{x\in(0,L)}\left(\int_{0}^{L-x}\rho_{1}(x)\,\ud x\right)\left(\int_{L-x}^{L}\left[\rho_{2}(x)\right]^{-1}\,\ud x\right)<\infty
$$
where $\ds Tf(x)=\int_{0}^{x}f(s)\,\ud x$. Moreover, the best constant $C$ in \eqref{PSWKVS19} satisfies $K\leq C\leq 2K$.
\end{lem}
Let $\beta$ such that $\ds\frac{1-\alpha_{j}}{2}<\beta<1$ for all $j=1,\ldots,N$ and $\delta_{j}$ are positive numbers that will be specified later. Following to Lemmas \ref{PSWKVS16} and \ref{PSWKVS18} and assumptions \eqref{IWKVS2} and \eqref{IWKVS3} then for $n$ large enough
\begin{align}\label{PSWKVS20}
\|v_{j,n}\|_{L^{2}\left(\left[\frac{\lambda_{n}^{-\delta_{j}}}{2},\lambda_{n}^{-\delta_{j}}\right]\right)}&\leq\max_{x\in\left[\frac{\lambda_{n}^{-\delta_{j}}}{2},\lambda_{n}^{-\delta_{j}}\right]}\left\{\frac{x^{1-\beta}}{d_{j}(x)^{\frac{1}{2}}}\right\}\left\|\frac{d_{j}^{\frac{1}{2}}}{x}(x^{\beta}v_{j,n})\right\|_{L^{2}\left(\left[\frac{\lambda_{n}^{-\delta_{j}}}{2},\lambda_{n}^{-\delta_{j}}\right]\right)}\nonumber
\\
&\leq C\lambda_{n}^{-\delta_{j}(1-\beta-\frac{\alpha_{j}}{2})}\left(\left\|d_{j}^{\frac{1}{2}}x^{\beta}v_{j,n}'\right\|_{L^{2}([0,\ell_{j}])}+\beta\left\|d_{j}^{\frac{1}{2}}x^{\beta-1}v_{j,n}\right\|_{L^{2}([0,\ell_{j}])}\right)\nonumber
\\
&\leq C\lambda_{n}^{-\delta_{j}(1-\beta-\frac{\alpha_{j}}{2})}\left\|d_{j}^{\frac{1}{2}}v_{j,n}'\right\|_{L^{2}([0,\ell_{j}])}.
\end{align}
Performing the following calculation and uses \eqref{PSWKVS20} one finds
\begin{multline}\label{PSWKVS21}                           
\min_{x\in\left[\frac{\lambda_{n}^{-\delta_{j}}}{2},\lambda_{n}^{-\delta_{j}}\right]}\left\{|v_{j,n}(x)|+|T_{j,n}(x)|\right\}\leq\sqrt{2}\lambda_{n}^{\frac{\delta_{j}}{2}}\left(\|v_{j,n}\|_{L^{2}\left(\left[\frac{\lambda_{n}^{-\delta_{j}}}{2},\lambda_{n}^{-\delta_{j}}\right]\right)}+\|T_{j,n}\|_{L^{2}\left(\left[\frac{\lambda_{n}^{-\delta_{j}}}{2},\lambda_{n}^{-\delta_{j}}\right]\right)}\right)
\\
\leq\sqrt{2}\lambda_{n}^{\frac{\delta_{j}}{2}}\left(\|v_{j,n}\|_{L^{2}\left(\left[\frac{\lambda_{n}^{-\delta_{j}}}{2},\lambda_{n}^{-\delta_{j}}\right]\right)}+\|u_{j,n}'\|_{L^{2}\left(\left[\frac{\lambda_{n}^{-\delta_{j}}}{2},\lambda_{n}^{-\delta_{j}}\right]\right)}+\|d_{j}v_{j,n}'\|_{L^{2}\left(\left[\frac{\lambda_{n}^{-\delta_{j}}}{2},\lambda_{n}^{-\delta_{j}}\right]\right)}\right)
\\
\leq C\lambda_{n}^{\frac{\delta_{j}}{2}}\Bigg(\lambda_{n}^{-\delta_{j}(1-\beta-\frac{\alpha_{j}}{2})}\left\|d_{j}^{\frac{1}{2}}v_{j,n}'\right\|_{L^{2}([0,\ell_{j}])}+\max_{x\in\left[\frac{\lambda_{n}^{-\delta_{j}}}{2},\lambda_{n}^{-\delta_{j}}\right]}\{d_{j}(x)^{-\frac{1}{2}}\}\|d_{j}^{\frac{1}{2}}u_{j,n}'\|_{L^{2}([0,\ell_{j}])}
\\
+\max_{x\in\left[\frac{\lambda_{n}^{-\delta_{j}}}{2},\lambda_{n}^{-\delta_{j}}\right]}\{d_{j}(x)^{\frac{1}{2}}\}\|d_{j}^{\frac{1}{2}}v_{j,n}'\|_{L^{2}([0,\ell_{j}])}\Bigg)
\\
\leq C\lambda_{n}^{\frac{\delta_{j}}{2}}\left(\lambda_{n}^{-\delta_{j}(1-\beta-\frac{\alpha_{j}}{2})}\left\|d_{j}^{\frac{1}{2}}v_{j,n}'\right\|_{L^{2}([0,\ell_{j}])}+\lambda_{n}^{\frac{\delta_{j}\alpha_{j}}{2}}\|d_{j}^{\frac{1}{2}}u_{j,n}'\|_{L^{2}([0,\ell_{j}])}+\lambda_{n}^{-\frac{\delta_{j}\alpha_{j}}{2}}\|d_{j}^{\frac{1}{2}}v_{j,n}'\|_{L^{2}([0,\ell_{j}])}\right).
\end{multline}
Inserting \eqref{PSWKVS11}, \eqref{PSWKVS12} into \eqref{PSWKVS21}, then we obtain
$$
\min_{x\in\left[\frac{\lambda_{n}^{-\delta_{j}}}{2},\lambda_{n}^{-\delta_{j}}\right]}\left\{|v_{j,n}(x)|+|T_{j,n}(x)|\right\}=\left(\lambda_{n}^{-\delta_{j}(\frac{1-\alpha_{j}}{2}-\beta)-\frac{\gamma}{2}}+\lambda_{n}^{\frac{\delta_{j}}{2}(\alpha_{j}+1)-\frac{\gamma}{2}-1}+\lambda_{n}^{\frac{\delta_{j}}{2}(1-\alpha_{j})-\frac{\gamma}{2}}\right)o(1),
$$
for every $\beta$ such that $\ds\frac{1-\alpha_{j}}{2}<\beta<1$ for every $j=1,\ldots,N$. Then we can choose $\beta$ such that $\ds-\delta_{j}(\frac{1-\alpha_{j}}{2}-\beta)-\frac{\gamma}{2}<0$ for every $j=1,\ldots,N$. Hence, as long as we choose $\delta_{j}>0$ and $\gamma>0$ such that
\begin{equation}\label{PSWKVS22}
\frac{\delta_{j}}{2}(\alpha_{j}+1)-\frac{\gamma}{2}-1\leq0\qquad\text{and}\qquad\frac{\delta_{j}}{2}(1-\alpha_{j})-\frac{\gamma}{2}\leq0\quad\forall\,j=1,\ldots,N.
\end{equation}
the following estimate holds
\begin{equation}\label{PSWKVS23}
\min_{x\in\left[\frac{\lambda_{n}^{-\delta_{j}}}{2},\lambda_{n}^{-\delta_{j}}\right]}\left\{|v_{j,n}(x)|+|T_{j,n}(x)|\right\}=o(1).
\end{equation}
And consequently, we are able to find $\xi_{j,n}\in\left[\frac{\lambda_{n}^{-\delta_{j}}}{2},\lambda_{n}^{-\delta_{j}}\right]$ such that
\begin{equation}\label{PSWKVS24}
|v_{j,n}(\xi_{j,n})|=o(1)\qquad\text{and}\qquad |T_{j,n}(\xi_{j,n})|=o(1).
\end{equation}

We set
$$
z_{j,n}^{\pm}(x)=\frac{i\lambda_{n}}{\sqrt{1+\lambda_{n}d_{j}(x)}}\int_{x}^{\xi_{j,n}}v_{j,n}(\tau)\,\ud\tau\pm v_{j,n}(x),\qquad\forall\,x\in[0,\xi_{j,n}].
$$
Then we have
\begin{multline}\label{PSWKVS25}
z_{j,n}^{\pm\;\prime}(x)=-\frac{i\lambda_{n}d_{j}'(x)}{4(1+\lambda_{n}d_{j}(x))}(z_{j,n}^{+}(x)+z_{j,n}^{-}(x))\mp\frac{i\lambda_{n}}{\sqrt{1+\lambda_{n}d_{j}(x)}}z_{j,n}^{\pm}(x)
\\
\mp\frac{\lambda_{n}^{2}}{1+\lambda_{n}d_{j}(x)}\int_{x}^{\xi_{j,n}}v_{j,n}(\tau)\,\ud\tau\pm v_{j,n}'(x).
\end{multline}
Combining \eqref{PSWKVS5} and \eqref{PSWKVS8} we have
\begin{align}\label{PSWKVS26}
\pm v_{j,n}'(x)&=\pm i\lambda_{n} u_{j,n}'(x)\mp\lambda_{n}^{-\gamma}f_{j,n}'(x)=\pm\frac{i\lambda_{n}}{1+i\lambda_{n}d_{j}(x)}T_{j,n}(x)\pm\frac{i\lambda_{n}^{1-\gamma}d_{j}(x)}{1+i\lambda_{n}d_{j}}f_{j,n}'(x)\mp\lambda_{n}^{-\gamma}f_{j,n}'(x)\nonumber
\\
&=\pm\frac{i\lambda_{n}}{1+i\lambda_{n}d_{j}}T_{j,n}(x)\mp\frac{\lambda_{n}^{-\gamma}}{1+i\lambda_{n}d_{j}(x)}f_{j,n}'(x).
\end{align}
Integrating \eqref{PSWKVS7} over $(x,\xi_{j,n})$ and multiplying by $\ds\pm\frac{i\lambda_{n}}{1+i\lambda_{n}d_{j}(x)}$ then we get
\begin{equation}\label{PSWKVS27}
\mp\frac{\lambda_{n}^{2}}{1+i\lambda_{n}d_{j}(x)}\int_{x}^{\xi_{j,n}}v_{j,n}(\tau)\,\ud\tau=\pm\frac{i\lambda_{n}}{1+i\lambda_{n}d_{j}(x)}(T_{j,n}(\xi_{j,n})-T_{j,n}(x))\pm\frac{i\lambda_{n}^{1-\gamma}}{1+i\lambda_{n}d_{j}(x)}\int_{x}^{\xi_{j,n}}g_{j,n}(\tau)\,\ud\tau.
\end{equation}
Inserting \eqref{PSWKVS26} and \eqref{PSWKVS27} into \eqref{PSWKVS25}, we find
\begin{equation}\label{PSWKVS28}
z_{j,n}^{\pm\;\prime}(x)=\mp\frac{i\lambda_{n}}{\sqrt{1+\lambda_{n}d_{j}(x)}}z_{j,n}^{\pm}(x)-\frac{i\lambda_{n}d_{j}'(x)}{4(1+\lambda_{n}d_{j}(x))}(z_{j,n}^{+}(x)+z_{j,n}^{-}(x))\pm\frac{i\lambda_{n}}{1+i\lambda_{n}d_{j}(x)}T_{j,n}(\xi_{j,n})\pm F_{j,n}(x)
\end{equation}
where 
$$
F_{j,n}(x)=-\frac{\lambda_{n}^{-\gamma}}{1+i\lambda_{n}d_{j}(x)}f_{j,n}'(x)+\frac{i\lambda_{n}^{1-\gamma}}{1+i\lambda_{n}d_{j}(x)}\int_{x}^{\xi_{j,n}}g_{j,n}(\tau)\,\ud\tau.
$$
Solving \eqref{PSWKVS28}, then for every $x\in[0,\xi_{j,n}]$, one gets
\begin{multline}\label{PSWKVS29}
z_{j,n}^{\pm}(x)=v_{j,n}(\xi_{j,n})e^{\mp(q_{j,n}(x)-q_{j,n}(\xi_{j,n}))}-\int_{\xi_{j,n}}^{x}e^{\mp(q_{j,n}(x)-q_{j,n}(s))}\frac{i\lambda_{n}d_{j}'(s)}{4(1+\lambda_{n}d_{j}(s))}(z_{j,n}^{+}(s)+z_{j,n}^{-}(s))\,\ud s
\\
\pm T_{j,n}(\xi_{j,n})\int_{\xi_{j,n}}^{x}e^{\mp(q_{j,n}(x)-q_{j,n}(s))}\frac{i\lambda_{n}}{1+i\lambda_{n}d_{j}(s)}\ud s\pm\int_{\xi_{j,n}}^{x}e^{\mp(q_{j,n}(x)-q_{j,n}(s))}F_{j,n}(s)\,\ud s
\end{multline}
where
$$
q_{j,n}(x)=i\lambda_{n}\int_{0}^{x}\frac{\ud s}{\sqrt{1+i\lambda_{n}d_{j}(s)}}.
$$

For all $x\in[0,\xi_{j,n}]$ we have
$$
q_{j,n}(x)=i\lambda_{n}\int_{0}^{x}\frac{e^{i\varphi_{j,n}(s)}}{(1+(\lambda_{j,n}d_{j}(s))^{2})^{\frac{1}{4}}}\,\ud s
$$
where
$$
\varphi_{j,n}(s)=-\frac{1}{2}\arg(1+i\lambda_{n}d_{j,n}(s)).
$$
Consequently,
$$
\re(q_{j,n}(x))=\pm\lambda_{n}\int_{0}^{x}\frac{\sin(\varphi_{j,n}(s))}{(1+(\lambda_{j,n}d_{j}(s))^{2})^{\frac{1}{4}}}\,\ud s.
$$
When $s\in[0,\xi_{j,n}]$ and from assumption we have
\begin{equation}\label{PSWKVS30}
\frac{1}{(1+(\lambda_{j,n}d_{j}(s))^{2})^{\frac{1}{4}}}=\left\{\begin{array}{ll}
O(1)&\text{if }\delta_{j}\alpha_{j}\geq 1
\\
O(\lambda_{n}^{\frac{\delta_{j}\alpha_{j}-1}{2}})&\text{if }\delta_{j}\alpha_{j}<1
\end{array}\right.
\end{equation}
and
\begin{align}\label{PSWKVS31}
|\sin(\varphi_{j,n}(s))|&=\sqrt{\frac{1}{2}-\frac{1}{2\sqrt{1+(\lambda_{n}d_{j}(s))^{2}}}}\nonumber
\\
&=\left\{\begin{array}{ll}
O(\lambda_{n}^{1-\delta_{j}\alpha_{j}})&\text{if }\delta_{j}\alpha_{j}>1
\\
O(1)&\text{if }\delta_{j}\alpha_{j}\leq1.
\end{array}\right.
\end{align}
With $0\leq x\leq s\leq\xi_{j,n}$, from \eqref{PSWKVS30} and \eqref{PSWKVS31} we see that
\begin{align*}
|\re(q_{j,n}(x)-q_{j,n}(s))|&\leq\lambda_{n}\int_{x}^{s}\frac{|\sin(\varphi_{j,n}(\tau))|}{(1+(\lambda_{n}d_{j}(\tau))^{2})^{\frac{1}{4}}}\ud\tau\nonumber
\\
&\leq\sup_{\tau\in(x,s)}\left\{\frac{|\sin(\varphi_{j,n}(\tau))|}{(1+(\lambda_{n}d_{j}(\tau))^{2})^{\frac{1}{4}}}\right\}\lambda_{n}(s-x)\nonumber
\\
&=\left\{\begin{array}{ll}
O(\lambda_{n}^{-(\alpha_{j}+1)\delta_{j}+2})=o(1)&\text{if }\delta_{j}\alpha_{j}>1
\\
O(\lambda_{n}^{1-\delta_{j}})=O(\lambda_{n}^{1-\frac{1}{\alpha_{j}}})=o(1)&\text{if }\delta_{j}\alpha_{j}=1
\\
O(\lambda_{n}^{\frac{\delta_{j}(\alpha_{j}-2)+1}{2}})&\text{if }\delta_{j}\alpha_{j}<1.
\end{array}\right.
\end{align*}
This implies that
\begin{align}\label{PSWKVS32}
|e^{\pm(q_{j,n}(x)-q_{j,n}(s))}|\leq 1
\end{align}
providing that $\delta_{j}>0$ satisfying \eqref{PSWKVS22} and
\begin{equation}\label{PSWKVS33}
\delta_{j}\geq\frac{1}{\alpha_{j}}\quad\text{or}\quad\delta_{j}\leq\frac{1}{2-\alpha_{j}}\qquad \forall\, j=1,\ldots,N.
\end{equation}
From \eqref{PSWKVS30} for every $x\in[0,\xi_{j,n}]$, we have
\begin{align}\label{PSWKVS34}
\int_{x}^{\xi_{j,n}}\frac{\ud s}{|1+i\lambda_{n}d_{j}(s)|}&=\left\{\begin{array}{ll}
O(1)(\xi_{j,n}-x)&\text{if }\delta_{j}\alpha_{j}\geq1
\\
O(\lambda_{n}^{\alpha_{j}\delta_{j}-1})(\xi_{j,n}-x)&\text{if }\delta_{j}\alpha_{j}<1
\end{array}\right.\nonumber
\\
&=\left\{\begin{array}{ll}
O(\lambda_{n}^{-\delta_{j}})&\text{if }\delta_{j}\alpha_{j}\geq1
\\
O(\lambda_{n}^{(\alpha_{j}-1)\delta_{j}-1})&\text{if }\delta_{j}\alpha_{j}<1
\end{array}\right.
\end{align}
and
\begin{align}\label{PSWKVS38}
\left(\int_{x}^{\xi_{j,n}}\frac{\ud s}{|1+i\lambda_{n}d_{j}(s)|^{2}}\right)^{\frac{1}{2}}=\left\{\begin{array}{ll}
O(\lambda_{n}^{-\frac{\delta_{j}}{2}})&\text{if }\delta_{j}\alpha_{j}\geq1
\\
O(\lambda_{n}^{(\alpha_{j}-\frac{1}{2})\delta_{j}-1})&\text{if }\delta_{j}\alpha_{j}<1.
\end{array}\right.
\end{align}
Therefore, when $\delta_{j}$ and $\gamma$ satisfy \eqref{PSWKVS22} and \eqref{PSWKVS33}, form \eqref{PSWKVS24}, \eqref{PSWKVS32} and \eqref{PSWKVS34} we obtain
\begin{align}\label{PSWKVS35}
|T_{j,n}(\xi_{j,n})|.\left|\int_{\xi_{j,n}}^{x}e^{\mp(q_{j,n}(x)-q_{j,n}(s))}\frac{i\lambda_{n}}{1+i\lambda_{n}d_{j}(s)}\ud s\right|&\leq|T_{j,n}(\xi_{j,n})|.\int_{x}^{\xi_{j,n}}\frac{\lambda_{n}}{|1+i\lambda_{n}d_{j}(s)|}\ud s\nonumber
\\
&=o(1)\quad\forall\,x\in[0,\xi_{j,n}].
\end{align}
Moreover, under the same conditions on $\delta_{j}$ and $\gamma$, due to \eqref{PSWKVS5}, \eqref{PSWKVS32} and \eqref{PSWKVS38} the Cauchy-Schwarz inequality leads to
\begin{align}\label{PSWKVS36}
\left|\int_{\xi_{j,n}}^{x}e^{\mp(q_{j,n}(x)-q_{j,n}(s))}\frac{\lambda_{n}^{-\gamma}}{1+i\lambda_{n}d_{j}(s)}f_{j,n}'(s)\,\ud s\right|&\leq\int_{x}^{\xi_{j,n}}\frac{\lambda_{n}^{-\gamma}}{|1+i\lambda_{n}d_{j}(s)|}|f_{j,n}'(s)|\,\ud s\nonumber
\\
&\leq\lambda_{n}^{-\gamma}\|f_{j,n}'\|_{L^{2}(0,\ell_{j})}\left(\int_{x}^{\xi_{j,n}}\frac{\,\ud s}{|1+i\lambda_{n}d_{j}(s)|^{2}}\right)^{\frac{1}{2}}\nonumber
\\
&=\left\{\begin{array}{ll}
o(\lambda_{n}^{-\frac{\delta_{j}}{2}-\gamma})&\text{if }\delta_{j}\alpha_{j}\geq 1
\\
o(\lambda_{n}^{\delta_{j}(\alpha_{j}-\frac{1}{2})-\gamma-1})&\text{if }\delta_{j}\alpha_{j}<1
\end{array}\right.\nonumber
\\
&=o(1)\quad\forall\,x\in[0,\xi_{j,n}],
\end{align}
and due to \eqref{PSWKVS7}, \eqref{PSWKVS32} and \eqref{PSWKVS34}, we have
\begin{align}\label{PSWKVS37}
\int_{\xi_{j,n}}^{x}&e^{\mp(q_{j,n}(x)-q_{j,n}(s))}\frac{i\lambda_{n}^{1-\gamma}}{1+i\lambda_{n}d_{j}(s)}\int_{s}^{\xi_{j,n}}g_{j,n}(\tau)\,\ud\tau\,\ud s\nonumber
\\
&=\left|\int_{x}^{\xi_{j,n}}\int_{x}^{\tau}e^{\mp(q_{j,n}(x)-q_{j,n}(s))}\frac{\lambda_{n}^{1-\gamma}}{1+i\lambda_{n}d_{j}(s)}g_{j,n}(\tau)\,\ud\tau\,\ud s\right|\nonumber
\\
&\leq\lambda_{n}^{1-\gamma}\int_{x}^{\xi_{j,n}}\int_{x}^{\tau}\frac{|g_{j,n}(\tau)|}{|1+i\lambda_{n}d_{j}(s)|}\,\ud\tau\,\ud s\nonumber
\\
&\leq\lambda_{n}^{1-\gamma}\int_{x}^{\xi_{j,n}}\frac{\ud s}{|1+i\lambda_{n}d_{j}(s)|}\int_{x}^{\xi_{j,n}}|g_{j,n}(\tau)|\,\ud\tau\nonumber
\\
&=\left\{\begin{array}{ll}
o(\lambda_{n}^{1-\gamma-\frac{3\delta_{j}}{2}})&\text{if }\delta_{j}\alpha_{j}\geq 1
\\
o(\lambda_{n}^{\delta_{j}(\alpha_{j}-\frac{3}{2})-\gamma})&\text{if }\delta_{j}\alpha_{j}<1
\end{array}\right.\nonumber
\\
&=o(1)\quad\forall\,x\in[0,\xi_{j,n}].
\end{align}
Combining \eqref{PSWKVS36} and \eqref{PSWKVS37} yields
\begin{equation}\label{PSWKVS39}
\left|\int_{\xi_{j,n}}^{x}e^{\mp(q_{j,n}(x)-q_{j,n}(s))}F_{j,n}(s)\,\ud s\right|=o(1)\quad\forall\,x\in[0,\xi_{j,n}].
\end{equation}
From assumption \eqref{IWKVS2} we have $d_{j}'(s)>0$ near $0$, then by Cauchy-Schwarz inequality and assumption \eqref{IWKVS3} we find
\begin{align}\label{PSWKVS40}
\int_{x}^{\xi_{j,n}}\frac{\lambda_{n}d_{j}'(s)}{|1+i\lambda_{n}d_{j}(s)|}\ud s&\leq\left(\int_{x}^{\xi_{j,n}}\frac{\lambda_{n}^{2}d_{j}'(s)}{1+(\lambda_{n}d_{j}(s))^{2}}\ud s\right)^{\frac{1}{2}}.\left(\int_{x}^{\xi_{j,n}}d_{j}'(s)\,\ud s\right)^{\frac{1}{2}}\nonumber
\\
&\leq\left(\arctan(\lambda_{n}^{2}d_{j}(\xi_{j,n}))-\arctan(\lambda_{n}^{2}d_{j}(x))\right)^{\frac{1}{2}}.\left(d_{j}(\xi_{j,n})-d_{j}(x)\right)^{\frac{1}{2}}\nonumber
\\
&=O(\lambda_{n}^{-\frac{\alpha_{j}\delta_{j}}{2}})\quad\forall\,x\in[0,\xi_{j,n}].
\end{align}
Inserting \eqref{PSWKVS24}, \eqref{PSWKVS32}, \eqref{PSWKVS35}, \eqref{PSWKVS39} and \eqref{PSWKVS40} into \eqref{PSWKVS29}, then we get
$$
|z_{j,n}^{\pm}(x)|\leq o(1)(m_{j,n}+1)\quad\forall\,x\in[0,\xi_{j,n}],
$$
where
$$
m_{j,n}=\max_{x\in[0,\xi_{j,n}]}\{|z_{j,n}^{+}(x)|+|z_{j,n}^{-}(x)|\},
$$
which leads to
\begin{equation}\label{PSWKVS41}
m_{j,n}=o(1).
\end{equation}

Since we can write
$$
v_{j,n}(x)=\frac{1}{2}(z_{j,n}^{+}(x)-z_{j,n}^{-}(x))
$$
then we follow from \eqref{PSWKVS41} that
\begin{equation}\label{PSWKVS42}
\|v_{j,n}\|_{L^{2}(0,\xi_{j,n})}=\frac{1}{2}\|z_{j,n}^{+}-z_{j,n}^{-}\|_{L^{2}(0,\xi_{j,n})}\leq\frac{m_{j,n}}{2}\sqrt{\xi_{j,n}}=o(\lambda_{n}^{-\frac{\delta_{j}}{2}})
\end{equation}
and
\begin{equation}\label{PSWKVS43}
|v_{j,n}(0)|=\frac{1}{2}|z_{j,n}^{+}(0)-z_{j,n}^{-}(0)|\leq \frac{m_{j,n}}{2}=o(1).
\end{equation}
Integrating \eqref{PSWKVS7} over $(0,\xi_{j,n})$,
\begin{equation}\label{PSWKVS44}
i\lambda_{n}\int_{0}^{\xi_{j,n}}v_{j,n}(s)\,\ud s-T_{j,n}(\xi_{j,n})+T_{j,n}(0)=\lambda_{n}^{-\gamma}\int_{0}^{\xi_{j,n}}g_{j,n}(s)\,\ud s.
\end{equation}
Due to \eqref{PSWKVS41} and the fact that $a(0)=0$, we have
\begin{equation}\label{PSWKVS45}
\left|i\lambda_{n}\int_{0}^{\xi_{j,n}}v_{j,n}\,\ud s\right|=\frac{1}{2}|z_{j,n}^{+}(0)-z_{j,n}^{-}(0)|=o(1).
\end{equation}
Substituting \eqref{PSWKVS7}, \eqref{PSWKVS23} and \eqref{PSWKVS45} into \eqref{PSWKVS44} yields
\begin{equation}\label{PSWKVS46}
|T_{j,n}(0)|=o(1).
\end{equation}

Substituting  \eqref{PSWKVS43}, \eqref{PSWKVS46} into \eqref{PSWKVS13}, by the transmission conditions \eqref{PSWKVS57} and \eqref{PSWKVS58} we conclude
\begin{equation}\label{PSWKVS47}
\int_{0}^{\ell_{0}}\left(|u_{0,n}'|^{2}+|v_{0,n}|^{2}\right)\,\ud x=o(1).
\end{equation}
For $j=1,\ldots,N$ we have
\begin{align}\label{PSWKVS48}
\|d_{j}^{\frac{1}{2}}u'_{j}\|_{L^{2}(0,\ell_{j})}&\geq\|d_{j}^{\frac{1}{2}}u'_{j}\|_{L^{2}(\xi_{j,n},\ell_{j})}\geq \min_{x\in(\xi_{j,n},\ell_{j})}\left(\sqrt{d_{j}(x)}\right)\|u_{j,n}'\|_{L^{2}(\xi_{j,n},\ell_{j})}\nonumber
\\
&\geq\sqrt{d_{j}(\xi_{j,n})}\|u_{j,n}'\|_{L^{2}(\xi_{j,n},\ell_{j})}\geq C\xi_{j,n}^{\frac{\alpha_{j}}{2}}\|u_{j,n}'\|_{L^{2}(\xi_{j,n},\ell_{j})}\geq C\lambda_{n}^{-\frac{\delta_{j}\alpha_{j}}{2}}\|u_{j,n}'\|_{L^{2}(\xi_{j,n},\ell_{j})}.
\end{align}
From \eqref{PSWKVS22} we have $\ds\frac{\alpha_{j}\delta_{j}}{2}\leq\frac{\gamma}{2}+1$ for every $j=1,\ldots,N$, then by combining \eqref{PSWKVS12} and \eqref{PSWKVS48} one gets
\begin{equation}\label{PSWKVS49}
\|u_{j,n}'\|_{L^{2}(\xi_{j,n},\ell_{j})}=o(\lambda_{n}^{\frac{\alpha_{j}\delta_{j}}{2}-\frac{\gamma}{2}-1}).
\end{equation}
Therefore by the trace formula
\begin{equation}\label{PSWKVS51}
|u_{j,n}(\xi_{j,n})|=o(\lambda_{n}^{\frac{\alpha_{j}\delta_{j}}{2}-\frac{\gamma}{2}-1}).
\end{equation}
By trace formula and \eqref{PSWKVS8} we have
\begin{equation}\label{PSWKVS54}
|u_{j,n}(0)|\leq C\|u_{j,n}'\|_{L^{2}(0,\ell_{j})}=O(\lambda_{n}^{\frac{\alpha_{j}\delta_{j}}{2}-\frac{\gamma}{2}-1}).
\end{equation}
From \eqref{PSWKVS8} and \eqref{PSWKVS5} one has
\begin{equation}\label{PSWKVS50}
\lambda_{n}\|u_{j,n}\|_{L^{2}(0,\ell_{j})}=O(1).
\end{equation}
Multiplying \eqref{PSWKVS7} by $\lambda_{n}^{-\gamma}\overline{u}_{j,n}$ and integrating over $(0,\xi_{j,n})$ then by integrating by parts we arrive
\begin{multline}\label{PSWKVS52}
i\lambda_{n}\langle v_{j,n},u_{j,n}\rangle_{L^{2}(0,\xi_{j,n})}+\langle d_{j,n}v_{j,n},u_{j,n}\rangle_{L^{2}(0,\xi_{j,n})}+\|u_{j,n}'\|_{L^{2}(0,\xi_{j,n})}
\\
+T_{j,n}(0)\overline{u}_{j,n}(0)-T_{j,n}(\xi_{j,n})\overline{u}_{j,n}(\xi_{j,n})=o(\lambda_{n}^{-\gamma}).
\end{multline}
Inserting \eqref{PSWKVS24}, \eqref{PSWKVS42}, \eqref{PSWKVS46}, \eqref{PSWKVS51}, \eqref{PSWKVS54} and \eqref{PSWKVS50} into \eqref{PSWKVS52} one finds
\begin{equation}\label{PSWKVS53}
\|u_{j,n}'\|_{L^{2}(0,\xi_{j,n})}=o(1),\quad\forall\,j=1,\ldots,N.
\end{equation}
So that, adding \eqref{PSWKVS49} and \eqref{PSWKVS53}, we obtain
\begin{equation}\label{PSWKVS55}
\|u_{j,n}'\|_{L^{2}(0,\ell_{j})}=o(1),\quad\forall\,j=1,\ldots,N.
\end{equation}
From \eqref{PSWKVS17}, \eqref{PSWKVS47} and \eqref{PSWKVS55} leads to
\begin{equation}\label{PSWKVS56}
\|v_{j}\|_{L^{2}(0,\ell_{j})}=o(1),\quad\forall\,j=1,\ldots,N.
\end{equation}
This conclude the prove since now we have proved that $\|(u_{n},v_{n})\|_{\mathcal{H}}=o(1)$ from \eqref{PSWKVS47}, \eqref{PSWKVS55} and \eqref{PSWKVS56} providing that \eqref{PSWKVS22} and \eqref{PSWKVS33} hold true.

Finally, noting that the best $\gamma$ and $\delta_{j}$, in term of maximization of $\gamma$, that satisfies \eqref{PSWKVS22} and \eqref{PSWKVS33} are where $\ds\gamma=\max\left\{\frac{1-\alpha_{j}}{2-\alpha_{j}},\;j=1,\ldots,N\right\}$ and $\ds\delta_{j}=\frac{1}{2-\alpha_{j}}$ for $j=1,\ldots,N$. This completes the proof.

\subsubsection*{Acknowledgments}
The author thanks professor Ka\"is Ammari for his advices, his suggestions and for all the discussions.

\end{document}